\documentclass[11pt]{article} 
\usepackage{amsfonts,amssymb,amsmath,amsthm,mathrsfs}
\usepackage{color}
\usepackage{tikz} 
\usepackage{hyperref} 
\usepackage{caption} 
\usepackage{marvosym}
\hypersetup{
   linkcolor=blue,          
   citecolor=blue,        
   filecolor=magenta,      
   urlcolor=cyan           
}

\usepackage[left=1in,right=1in,top=1.3in,bottom=1.4in]{geometry}
\usepackage{xspace}
\usepackage{enumitem}

\newcommand{\Z}{\mathbb{Z}}

\newcommand{\Sec}{\ensuremath{\S}}

\newcommand{\N}{\mathbb{N}}
\newcommand{\R}{\mathbb{R}}

\newcommand{\res}{\mathop{\hbox{\vrule height 7pt width .5pt depth 0pt \vrule height .5pt width 6pt depth 0pt}}\nolimits}

\newcommand{\hausd}{\mathcal H} 

\def\d{\,\mathrm{d}}
\newcommand{\insieme}[1]{\left \{#1\right \}}

\newcommand{\F}{\mathcal F}

\def\GtL{{\mathcal{G}}_{\tau,L}}
\def\Fj{\widetilde{\mathcal{F}}_{J,L}}

\def\FtL{\mathcal{F}_{\tau,L}}
\def\FzL{\mathcal{F}_{0,L}}

\def\lt{\left}
   \def\rt{\right}
\def\les{\lesssim}
\def\ges{\gtrsim}
\def\bCp{\overline{C}_q}
\def\GzL{{\mathcal{G}}_{0,L}}

\DeclareMathOperator{\spt}{spt}

\DeclareMathOperator{\per}{Per}

\newtheorem{theorem}{Theorem}[section]

\newtheorem{lemma}[theorem]{Lemma}
\newtheorem{definition}[theorem]{Definition}

\newtheorem{remark}[theorem]{Remark}

\newtheorem{proposition}[theorem]{Proposition}

\newtheorem{corollary}[theorem]{Corollary}

\title{On the optimality of stripes in a variational model with non-local interactions}
\date{}

\usepackage{authblk}

\author[1]{Michael Goldman\thanks{goldman@math.univ-paris-diderot.fr}}
\author[2]{Eris Runa\thanks{eris.runa@mis.mpg.de}}
\affil[1]{LJLL, Universit\'e Paris Diderot,  CNRS, UMR 7598,  France. }
\affil[2]{Max Planck Institut f\"ur Mathematik in den Naturwissenschaften, Leipzig}

\usepackage{esint}

\def\1{{\mathchoice {1\mskip-4mu\mathrm l}      
      {1\mskip-4mu\mathrm l} 
      {1\mskip-4.5mu\mathrm l} {1\mskip-5mu\mathrm l}}} 
 
\def\comment#1{} 
\newtheoremstyle{thm}{2ex}{2ex}{\itshape\rmfamily}{} 
{\bfseries\rmfamily}{}{1.7ex}{} 

\newtheoremstyle{rem}{1.3ex}{1.3ex}{\rmfamily}{} 
{\itshape\rmfamily}{}{1.5ex}{} 

\def\Pe{\per_1}

\def\XXint#1#2#3{{\setbox0=\hbox{$#1{#2#3}{\int}$} 
      \vcenter{\hbox{$#2#3$}}\kern-.5\wd0}}

\usepackage{microtype}

\begin{document}
\maketitle
\begin{abstract} 
   We study pattern formation for a  variational  model displaying competition  between a local term penalizing interfaces and a non-local term favoring oscillations. By means of a $\Gamma-$convergence analysis, 
   we show that as the parameter $J$ converges to a critical value $J_c$, the minimizers converge to periodic one-dimensional stripes. A similar analysis has been previously  performed by other authors for related discrete
   systems. In that context, a central point  is that  each ``angle'' comes with a  strictly positive contribution to the energy.  Since this is not anymore the case in the continuous setting,  we 
   need to overcome this difficulty by slicing arguments and  a rigidity result. 
\end{abstract} 

\section{Introduction} 
\label{sec:intro} 
Motivated by recent works \cite{2011PhRvB..84f4205G,2014CMaPh.tmp..127G,GiuSeirGS} on striped patterns in Ising models with competing interactions, we consider
for $d\ge 2$, $J,L>0$ and $p>2d$  the functional
\begin{equation}\label{functionalintro}
   \Fj(E):=\frac{1}{L^d}\lt(J \int_{\partial E\cap [0,L)^{d}} |\nu^E|_1 d\hausd^{d-1} -\int_{[0,L)^{d}\times\R^d}\frac{|\chi_E(x)-\chi_{E}(y)|}{|x-y|^p+1}dx dy \rt),
\end{equation}
where $E$ is a $[0,L)^d-$periodic set, $\nu^E$ is its external normal and  $|\cdot|_1$ denotes the $1-$norm. As in the discrete case \cite{2011PhRvB..84f4205G}, it can be shown (see Proposition \ref{prop:nonneg}) that for $J\ge J_c:=\int_{\R^d} \frac{|\zeta_1|}{1+|\zeta|^p}$, the energy 
is always non-negative and thus minimizers are the uniform states while for $J<J_c$, there exists non trivial minimizers. We are interested here in the behavior of these minimizers as $J\uparrow J_c$. 
Building on the computations made in \cite{2011PhRvB..84f4205G}, it is expected that for $\tau:=J_c-J$ small enough, minimizers are periodic striped patterns. A simple computation (see \eqref{explicitstripes}) shows that
the optimal stripes have width of order $\tau^{-1/(p-d-1)}$ and energy of order $-\tau^{(p-d)/(p-d-1)}$. This motivates the rescaling given in \eqref{changevar} which yields stripes of width and energy of order one as $\tau$ goes to zero. 
After this rescaling, we are led to study the  minimizers of 
\begin{multline*}
   \FtL(E):=\frac{1}{L^d}\lt(-\int_{\partial E\cap [0,L)^{d}} |\nu^E|_1 d\hausd^{d-1}\rt.\\
   \lt.+\int_{\R^d} \frac{1}{|\zeta|^p+\tau^{p/(p-d-1)}} \lt[\int_{\partial E \cap [0,L)^{d}} \sum_{i=1}^d|\nu^E_i| |\zeta_i| d\hausd^{d-1}-\int_{[0,L)^{d}}|\chi_E(x)-\chi_E(x+\zeta)|dx \rt] d\zeta\rt).
\end{multline*}

Our main theorem is a $\Gamma-$convergence \cite{braides} result for $\FtL$.
\begin{theorem}\label{Gammaintro}
   For $p>2d$ and $L>1$, the functionals $\FtL$ $\Gamma-$converge as $\tau$ goes to zero  with respect to the $L^1-$convergence to the functional defined for sets $E=\widehat{E}\times\R^{d-1}$
   where $\widehat{E}$  is $L-$periodic with $\sharp \{\partial \widehat{E}\cap [0,L)\}<+\infty$,  by  
   \begin{equation} 
      \label{eq:gammmaintro}
      \begin{split}
      \FzL(E):=
      \frac{1}{L}\lt( -\sharp \{\partial \widehat{E}\cap [0,L)\}+\int_{\R^d} \frac{1}{|\zeta|^p} \lt[\sum_{x\in \partial \widehat{E}\cap [0,L)} |\zeta_1|-\int_{0}^L|\chi_{\widehat{E}}(x)-\chi_{\widehat{E}}(x+\zeta_1)|dx \rt] d\zeta \rt),
      \end{split}
   \end{equation} 
   and $\FzL(E):=+\infty$ otherwise. Moreover, if $E^\tau$ is such that $\sup_{\tau} \FtL(E^\tau)<+\infty$, then up to a relabeling of the coordinate axes, 
   there is a subsequence which converges in $L^1$ to some set $E$ with $E=\widehat{E}\times \R^{d-1}$ and  $\sharp \{\partial \widehat{E}\cap [0,L)\}<+\infty$. 
\end{theorem}

This theorem is a restatement of Theorem \ref{prop:GammaLimit}. Using the method of reflection positivity, we can then compute the minimizers of the limiting energy (see Theorem \ref{theo:1d}).
\begin{theorem}\label{min1dintro}
   There exist $h^\star>0$ (whose value is given in Lemma \ref{lemmahstar}) and $C>0$ such that for every $L>1$, minimizers of $\FzL$ are periodic stripes of width $h$ with
   \[|h-h^\star|\le CL^{-1}.\]
\end{theorem}

As a direct consequence of Theorem \ref{Gammaintro} and Theorem \ref{min1dintro}, we obtain the following corollary
\begin{corollary}
   Let $E^\tau$ be minimizers of $\FtL$. Then, up to a subsequence, they converge as $\tau$ goes to zero to stripes of width $h$ with $ |h-h^\star|\le CL^{-1}$.
\end{corollary}
Let us notice that as pointed out in Remark \ref{remhunique}, for most values of $L$, the minimizer of $\FzL$ is unique. In this case, up to a rotation, the whole
sequence $E^\tau$ converges.\\

We now give an outline of the proofs of Theorem \ref{Gammaintro} and  Theorem \ref{min1dintro}, and discuss the relations and main differences  with those in the discrete setting \cite{2014CMaPh.tmp..127G,GiuSeirGS}. 
The main ingredients in the proof of  Theorem \ref{Gammaintro} are Lemma \ref{lemsplit} which permits to identify the part of the energy which penalizes non straight boundaries, the slicing formula \eqref{eq:sliceIdentity}, 
the crucial (but simple) estimate \eqref{toprovebandesim} which leads to the  estimate \eqref{eq:basicComparison21} of the non-local part of the energy
by the local widths and gaps and then to an estimate of the perimeter by the total energy (which gives  strong compactness), and finally a rigidity result (see Proposition \ref{rigidity}) which proves that in the limit, sets of finite energy must be one-dimensional. 
This last ingredient is based on the study of a functional which is somewhat reminiscent of integral characterizations of Sobolev spaces that have recently received a lot of attention  \cite{BBM,Bre,DMS}.
This connection will be further explored in a future work. As in  the proofs of \cite{2014CMaPh.tmp..127G,GiuSeirGS}, a central point is to estimate the cost of ``angles''. However, this is where the biggest difference
between the discrete and the continuous settings appears. In fact, in the discrete one, angles are quantized and thus carry a positive energy which forbids the presence of angles on a scale which is much larger than the typical width  of the stripes 
(see for instance \cite[Lemma~2]{GiuSeirGS}). On the contrary, in the continuous setting,
angles may be extremely small and give almost no contribution to the energy. Moreover, in the discrete case, the  number of ``angles'' gives an upper bound on the perimeter, which is again not the case in the continuous setting. 
The same observations also hold for the local widths and gaps (which in \cite{2014CMaPh.tmp..127G,GiuSeirGS} appear in the form of the distance between bonds facing each others). For all these reasons, it
seems difficult in our setting to estimate the contributions of the ``angles'' for $\tau>0$. This is the main reason why we need to pass to the limit $\tau\to 0$. Since we use a compactness argument (given by a $BV-$bound), 
a caveat of our approach is that it does not yield a rate of convergence of the minimizers of $\FtL$ to the optimal stripes. In particular, contrary to what is expected, this rate might depend on $L$. 
This can be compared with
the striking result of \cite{GiuSeirGS} where the authors were recently able to prove that for $\tau$ small enough but positive, periodic stripes are minimizers under compact perturbations of the  discrete energy in $\R^d$
\footnote{A. Giuliani has pointed out to us that from the proofs in \cite{GiuSeirGS}, it follows that stripes of width $h^\star$ are also minimizers under periodic boundary conditions in cubes of  arbitrary size proportional to $h^\star$.}. 
Yet, we believe that our approach based on slicing and on the splitting estimate \eqref{lowerboundIlem}, gives a good insight on some of the more combinatorial proofs of \cite{2014CMaPh.tmp..127G,GiuSeirGS}. Moreover, our proof easily extends  to more general
kernels behaving like $|\zeta|^{-p}$ at infinity or to model containing for instance volume constraints.
Let us add the technical observation that due to the minus sign in front of the perimeter in the definition of $\FtL$, the lower-bound \eqref{gammaliminf} can look surprising at first sight. In order to obtain it, we need to combine the bound \eqref{estiminf} on the gaps with
the fact that the limiting objects are one dimensional.
The proof of Theorem \ref{min1dintro} is based on reflection positivity \cite{fro} and does not differ much for instance from the proofs in \cite{giul_lieb_lebo_stripeddipole,MR2864796}. Since this technique is not so well known in the Calculus of Variation community and since  besides \cite{giul_lieb_lebo_meanfield,MR2864796}
we are not aware of many examples where
reflection positivity has been used in a continuous context, we decided to include the proof of Theorem \ref{min1dintro} for the reader's convenience. \\

Let us now comment on some choices we made in \eqref{functionalintro}. First, as in the discrete case, we are restricting ourselves to $p>2d$. This condition ensures that the energy
of stripes scales differently compared to the energy of a checkerboard (see \cite{2011PhRvB..84f4205G}). This is reflected in the fact that for the limiting functional, only striped patterns are admissible. 
Unfortunately, this leaves out the most interesting cases $d=3$ and $p=1$ or $p=3$, corresponding to Coulombic or Dipole interactions.  Second, we decided to work with the $1-$perimeter
instead of the usual Euclidean perimeter. This choice makes the splitting and slicing arguments work better by identifying the preferred axes of periodicity. The extension of this work to the classical perimeter will
be the subject of further investigations.\\  
Building on the results obtained in this paper,  it has been proven recently in \cite{2017arXiv170207334D} that for small enough $\tau$, the minimizers of $\FtL$ are periodic stripes.

The functional \eqref{functionalintro} is one of the simplest examples of a variational problem with competition between a local term penalizing interfaces and a repulsive non-local term. This competition leads to a complex pattern formation. 
The closest model to \eqref{functionalintro} is certainly the sharp interface version of the Ohta-Kawasaki functional which has been used to model diblock copolymers \cite{OhtaKawasaki,CicSpa,KnMu,Mu} or nuclear matter \cite{pasta}. 
Minimizers of this type of  variational problems often exhibit periodicity (see for instance \cite{chomawil} for some numerics). However, besides the one-dimensional situation \cite{Muller} (see also \cite{MoSter,SCN} for an almost one-dimensional case) and the low volume fraction limit \cite{ChoPe,GoMuSe,KnMuNo,BPT}
not much is known. To the best of our knowledge, the  only results available on periodicity of minimizers for intermediate volume fractions are  \cite{MR2864796} and the uniform local energy distribution \cite{Conti2006,ACO} as well as minimality in the perimeter dominant regime \cite{SterTop,AcFuMo,RenWei,Cris}. 
We refer to \cite{2014CMaPh.tmp..127G,GiuSeirGS} for more references in particular on the discrete setting and to the review paper \cite{BlancLewin} for a discussion of the related issue of crystallization.\\

The paper is organized as follows. In Section \ref{sec:not}, we set some notation and recall basic facts about sets of finite perimeter. In Section \ref{sec:func}, we derive the functional $\FtL$ from \eqref{functionalintro}
and prove some useful estimates. In Section \ref{sec:rigidity}, we prove the rigidity result Proposition \ref{rigidity}. Then, in Section \ref{sec:gamma}, we prove our $\Gamma-$convergence result. Finally in Section \ref{sec:1d}, we study the minimizers of the limiting problem.

\section{Notation}\label{sec:not}
In the paper we will use the following notation. The symbols $\sim$, $\ges$, $\les$ indicate estimates that hold up to a global constant depending only on the dimension. For instance, $f\les g$ denotes the existence of a constant $C>0$ such that $f\le Cg$,
$f\sim g$ means $f\les g$ and $g\les f$. We let $(e_1,\cdots, e_d)$ be the canonical basis of $\R^d$. For $(x,\zeta) \in (\R^d)^2$ and $i\in\{1,\ldots,d\}$, we let $x+\zeta_i:=x+\zeta_i e_i$, and then $x_i^\perp:= x- x_i$. We will denote by $|x|$ the Euclidean norm of $x$ and by $|x|_1:=\sum_{i=1}^d|x_i|$ its $1-$norm. For $L>0$, we will let  $Q_L:=[0,L)^d$. We let $\partial_i f:=\frac{\partial f}{\partial x_i}$. 
When it is clear from the context, we will not specify the measure of integration in the integrals. We take as a convention that whenever we integrate over $(x,\zeta)$ (respectively $(x,z)$),
the integral on the unbounded domain always concerns $\zeta$ (respectively $z$). For a $k-$dimensional set $E\subset \R^k$, we let $|E|$ be its Lebesgue measure. For $z\in \R$, we let $z_\pm$ be the positive and negative parts of $z$.   
\subsection{Sets of finite perimeter} 

The purpose of this section is to recall the definition  of sets of  finite perimeter. For a general introduction we refer to \cite{AFPBV}. Let us start with the one-dimensional case. We say that a $L-$periodic set $E\subset \R$ is of finite perimeter if $E\cap [0,L)=\cup_{i=1}^N I_i$ for some $N\in \N$ and some disjoint intervals $I_i$. We then let 
\[\per(E,[0,L)):=2N.\]
By periodicity, we will often assume that $E\cap [0,L]=\cup_{i=1}^N (s_i,t_i)$ with $s_1>0$ and $t_N<L$. For any interval $I$, we let $\per(E,I):=\sharp\{\partial E\cap I\}$.\\

We can now turn to higher dimension
\begin{definition} 
   \label{def:bv_and_finite_perimeter}
   A $Q_L-$periodic set $E$ is said to be of finite perimeter if $D\chi_{E}$, the distributional derivative of $\chi_E$, is a locally finite measure. For such a set, we let $\partial E$ be the collection of all points $x\in \spt (D\chi_{E})$ such that the limit
   \begin{equation*} 
      \begin{split}
         \nu^{E}(x):= -\lim_{\rho\downarrow 0} \frac{D \chi_{E}(B(x,\rho))}{|D\chi_{E}| (B(x,\rho))}
      \end{split}
   \end{equation*} 
   exists and satisfies $|\nu^{E}(x)|= 1$. We call $\nu^{E}$ the generalized outer normal to $E$. We then have $D\chi_E= -\nu^E\hausd^{d-1} \res \partial E$, where  $\hausd^{d-1}\res \partial E$ is the restriction of the $(d-1)-$dimensional Hausdorff measure to $\partial E$.
\end{definition} 

We then define 
\begin{equation*} 
   \begin{split}
      \per_{1}(E,Q_L):= \int_{\partial E\cap Q_L}|\nu^{E}(x)|_1 \d\hausd^{d-1}(x)=\int_{\partial E\cap Q_L}\sum_{i=1}^d|\nu_i^{E}(x)| \d\hausd^{d-1}(x).
   \end{split}
\end{equation*} 
 As for the one-dimensional case, by periodicity we will always assume that $|D\chi_E|(\partial Q_L)=0$ so that $\Pe(E,Q_L)$
coincides with the $1-$perimeter of $E$ in $(0,L)^d$. \\

For $i\in \{1,\ldots,d\}$, $x_{i}^\perp\in [0,L)^{d-1} $ and $E\subset Q_L$, we define the one-dimensional slices
\begin{equation*} 
   \begin{split}
      E_{x_{i}^\perp}:=\insieme{x_{i}\in [0,L):\ x_{i}+x_{i}^\perp\in E}.
   \end{split}
\end{equation*} 

Note that in the above definition there is an abuse of notation as the information on the direction of the slice is contained in the index $x^\perp_i$. As it would be always clear from the context which is the direction of the slicing, we hope this will not cause confusion to the reader.    

Then,  for every $i\in \{1,\ldots,d\}$, the following slicing formula holds (see ~\cite[\Sec~3.7]{AFPBV}) 
\begin{equation} 
   \label{eq:sliceIdentity}
   \begin{split}
      \int_{\partial E\cap Q_L} |\nu^E_i| = \int_{[0,L)^{d-1}} \per(E_{x_{i}^\perp},[0,L)).
   \end{split}
\end{equation} 


\section{The functional and preliminary results}\label{sec:func}
We recall that we are considering the functional 
\[
   \Fj(E):=\frac{1}{L^d}\lt(J \Pe(E,Q_L)-\int_{Q_L\times\R^d}K_1(x-y)|\chi_E(x)-\chi_{E}(y)| \rt),
\]
where $K_1(\zeta):=\frac{1}{|\zeta|^p+1}$ for some $p>2d$. More generally, for $\tau\ge 0$, we let $K_\tau(\zeta):=\frac{1}{|\zeta|^p+\tau^{{p}/{(p-d-1)}}}$. Notice that the functional $\Fj$ can also be written as 
\begin{equation}\label{defener1}
   \Fj(E)=\frac{1}{L^d}\lt(J \Pe(E,Q_L)-\int_{Q_L\times\R^d}K_1(\zeta)|\chi_E(x)-\chi_{E}(x+\zeta)| \rt).
\end{equation}
The aim of this section is to give some first properties of $\Fj$. We will in particular show that there exists a positive constant $J_c$ such that for $J>J_c$ all minimizers
of $\Fj$ are trivial while for $J<J_c$ they are not. This will lead us to study the behavior of these minimizers in term of the parameter $\tau:=J_c-J$ after suitable rescaling.\\
Let us point out that in this section, by approximation we can always  work with polygonal sets having  only horizontal and vertical edges.
For these sets, both the definition of $\Pe(E,Q_L)$ and \eqref{eq:sliceIdentity}  can be obtained without referring to the theory of sets of finite perimeter.
\begin{remark}\label{remper}
   Since $E$ is periodic,
   \[
      \int_{Q_L\times\R^d}K_1(x-y)|\chi_E(x)-\chi_{E}(y)| =2\int_{Q_L\times \R^d} K_1(x-y)\chi_E(x)\chi_{E^c}(y).
   \]
   Indeed, since $|\chi_E(x)-\chi_{E}(x+\zeta)|=\chi_E(x)\chi_{E^c}(x+\zeta)+\chi_E(x+\zeta)\chi_{E^c}(x)$, it follows from \eqref{defener1} and
   \begin{align*}
      \int_{Q_L\times \R^d}K_1(\zeta) \chi_E(x)\chi_{E^c}(x+\zeta)&=\int_{\R^d} K_1(\zeta)\int_{Q_L+\zeta} \chi_{E}(\widetilde{x}-\zeta)\chi_{E^c}(\widetilde{x})\\
      &=\int_{\R^d} K_1(\zeta)\int_{Q_L} \chi_{E}(\widetilde{x}-\zeta)\chi_{E^c}(\widetilde{x})\\
      &=\int_{Q_L\times \R^d}K_1(\zeta) \chi_{E^c}(x)\chi_{E}(x+\zeta).
   \end{align*}

\end{remark}

We recall that for $(x,\zeta) \in (\R^d)^2$, we let $x+\zeta_i:=x+\zeta_i e_i$ and $\zeta_i^\perp:= \zeta-\zeta_i$. 
\begin{lemma}\label{lemsplit}
   For every $Q_L-$periodic set $E$ and every $\tau>0$, it holds
   \begin{multline}\label{lowerboundIlem}
      \int_{Q_L\times\R^d} K_\tau(\zeta) |\chi_E(x)-\chi_E(x+\zeta)|\le \int_{Q_L\times \R^d} K_\tau(\zeta)\sum_{i=1}^d |\chi_E(x)-\chi_E(x+\zeta_i)|\\
      - \frac{2}{d}\int_{Q_L\times \R^d} K_\tau(\zeta) \sum_{i=1}^d|\chi_E(x)-\chi_{E}(x+\zeta_i)||\chi_E(x)-\chi_E(x+\zeta_i^\perp)|.
   \end{multline}

\end{lemma}
\begin{proof}
   We claim that for every $i\in \{1,\ldots,d\}$,
   \begin{multline}\label{lowerboundI}
      \int_{Q_L\times\R^d} K_\tau(\zeta) |\chi_E(x)-\chi_E(x+\zeta)|\le \int_{Q_L\times \R^d} K_\tau(\zeta) \sum_{i=1}^d |\chi_E(x)-\chi_E(x+\zeta_i)|\\
      - 2\int_{Q_L\times \R^d} K_\tau(\zeta) |\chi_E(x)-\chi_{E}(x+\zeta_i)||\chi_E(x)-\chi_E(x+\zeta_i^\perp)|.
   \end{multline}
   Summing \eqref{lowerboundI} over $i$ and dividing by $d$, would then yield \eqref{lowerboundIlem}. Without loss of generality, we may assume that  $i=1$. By disjunction of cases, it can be seen that for every $x,\zeta$,
   \begin{multline}\label{equaldisj}
      |\chi_E(x)-\chi_E(x+\zeta)|=|\chi_E(x)-\chi_{E}(x+\zeta_1)|+|\chi_E(x+\zeta_1)-\chi_{E}(x+\zeta)|\\
      - 2|\chi_E(x)-\chi_{E}(x+\zeta_1)||\chi_E(x+\zeta_1)-\chi_E(x+\zeta)|.
   \end{multline}
   We thus have by integration and using the periodicity of $E$ as in Remark \ref{remper},
   \begin{multline*}
      \int_{Q_L\times\R^d} K_\tau(\zeta) |\chi_E(x)-\chi_E(x+\zeta)|= \int_{Q_L\times \R^d} K_\tau(\zeta) (|\chi_E(x)-\chi_E(x+\zeta_1)|+ |\chi_E(x)-\chi_E(x+\zeta^\perp_1)|)\\
      - 2\int_{Q_L\times \R^d} K_\tau(\zeta) |\chi_E(x)-\chi_{E}(x+\zeta_1)||\chi_E(x)-\chi_E(x+\zeta_1^\perp)|.
   \end{multline*}
   Using that by the triangle inequality, 
   \[|\chi_E(x)-\chi_E(x+\zeta^\perp_1)|\le \sum_{k=2}^{d} \lt|\chi_E\lt(x+\sum_{j=2}^{k-1} \zeta_j\rt)-\chi_E\lt(x+\sum_{j=2}^k \zeta_j\rt)\rt|,\]
   and using again periodicity of $E$, we get \eqref{lowerboundI}. 

\end{proof}

\begin{remark}
   Using \eqref{equaldisj} recursively, one could get an equality in \eqref{lowerboundIlem} by replacing the term  $\int_{Q_L\times \R^d} K_\tau(\zeta) \sum_{i=1}^d|\chi_E(x)-\chi_{E}(x+\zeta_i)||\chi_E(x)-\chi_E(x+\zeta_i^\perp)|$ by a more complex one. 
   However, for our purpose this simpler bound is sufficient. Notice also that if $E$ depends only on one variable, then equality holds.  
\end{remark}

For a $L-$periodic set  $E=\cup_{i\in \Z}(s_i,t_i)\subset \R$ of finite perimeter,  we let for $i\in \Z$,
\[
   h(t_i):=h(s_i):=t_i-s_i
   \qquad g(t_i):= s_{i+1}-t_i \qquad g(s_i):=s_i-t_{i-1}.
\]
We then define for $z\in \R$, and $i\in \Z$,
\[
   \eta(t_i,z):=\min(z_+,h(t_i))+\min(z_-,g(t_i))\]
and
\[
   \eta(s_i,z):=\min(z_+,g(s_i))+\min(z_-,h(s_i)).
\]
For a $Q_L$-periodic set $E$ of finite perimeter, the functions $h_{x_i}$, $g_{x_i}$  and $\eta_{x_i}$ are defined by slicing. 
For instance, for $x_1\in E_{x_1^{\perp}}$ such that $\nu_1^E(x_1,x_1^\perp)>0$ (so that $x_1=t_i$ for some $i$) and $\zeta\in \R^d$,
\[
   \eta_{x_1^\perp}(x_1,\zeta_1):=\min((\zeta_1)_+,h_{x_1^\perp}(x_1))+\min((\zeta_1)_-,g_{x_1^\perp}(x_1)).\]

We may now prove a simple but crucial estimate
\begin{lemma}
   \label{lemma:upperboundBySplitting}
   For every $L-$periodic set $E$ of finite perimeter and every $z\in \R$, it holds
   \begin{equation}\label{toprovebandesim}
      \int_{0}^L |\chi_{E}(x)-\chi_{E}(x+z)| dx\le \sum_{x\in \partial E\cap [0,L)} \eta(x,z).
   \end{equation}
   By integration, for every $\tau\ge 0$, every $Q_L-$periodic set $E$ of finite perimeter, every $\zeta\in \R^d$ and $i\in \{1,\ldots,d\}$,
   \begin{equation}\label{bandestim}
      \int_{Q_L\times \R^d} K_\tau(\zeta) |\chi_{E}(x)-\chi_{E}(x+\zeta_i)|\le \int_{\partial E\cap Q_L} |\nu^E_i| \int_{\R^d} K_\tau(\zeta) \eta_{x_i^\perp}(x_i,\zeta_i).
   \end{equation}
\end{lemma}
\begin{proof}
   Let us prove \eqref{toprovebandesim}. We consider only the case $z\ge 0$ since the case $z\le 0$ can be treated in a similar way. Up to a translation, we can assume that $E\cap [0,L)= \cup_{i=1}^N (s_i,t_i)$ for some $N\in \N$. 
   We then have by periodicity of $E$,
   \begin{align*}
      \int_0^L|\chi_E(x)-\chi_E(x+z)|&=\sum_{i=1}^N \int_{s_i}^{t_i} \chi_{E^c}(x+z)+\int_0^{s_1} \chi_{E}(x+z)+\int_{t_N}^L \chi_E(x+z)+\sum_{i=2}^{N} \int_{t_{i-1}}^{s_{i}} \chi_{E}(x+z)\\
      &=\sum_{i=1}^N \int_{s_i}^{t_i} \chi_{E^c}(x+z)+\sum_{i=1}^{N} \int_{t_{i-1}}^{s_{i}} \chi_{E}(x+z).
   \end{align*}

   For every $i\in [1,N]$, if $x\in (s_i,t_i)$ and $x+z\in E^c$, then $x+z\ge t_i$ and thus $|x-t_i|\le  \min(z, h(t_i))=\eta(t_i,z)$. Similarly, for $i\in [1,N-1]$, if $x\in(t_{i-1},s_i)$ and $x+z\in E$, then $|x-s_i|\le \eta(s_i,z)$. Therefore,
   \[
      \int_{0}^L |\chi_{E}(x)-\chi_{E}(x+z)| dx\le \sum_{i=1}^{N} \eta(s_i,z)+\eta(t_i,z),
   \]
   which proves \eqref{toprovebandesim}.
\end{proof}

We now show that this quickly implies that for $J\ge J_c$ the minimizers of $\Fj$ are trivial. A somewhat similar proof in the discrete setting may be found in \cite{2011PhRvB..84f4205G}.
\begin{proposition}\label{prop:nonneg}
   For $J\ge J_c:=\int_{\R^d} K_1(\zeta)|\zeta_1|$ and every $Q_L-$periodic set $E$, $\Fj(E)\ge 0$.
\end{proposition}
\begin{proof}
   Putting Lemma \ref{lemsplit} together with \eqref{bandestim}, we get 
   \begin{align*}
      \Fj(E)&\ge \frac{1}{L^d}\lt(J \Pe(E,Q_L)-\sum_{i=1}^d\int_{\partial E\cap Q_L} |\nu^E_i| \lt[\int_{\R^d} K_1(\zeta) \eta_{x_i^\perp}(x_i,\zeta_i) \rt]\rt)\\
      &\ge \frac{1}{L^d} \sum_{i=1}^d\int_{\partial E\cap Q_L} |\nu^E_i| \lt[J-\int_{\R^d} K_1(\zeta)|\zeta_i|\rt]\\
      &=\frac{1}{L^d} \int_{\partial E\cap Q_L} |\nu^E|_1 \lt[J-\int_{\R^d} K_1(\zeta)|\zeta_1|\rt],
   \end{align*}
   which proves the claim.
\end{proof}

Letting $\tau:= J_c-J$ for $J<J_c$, it is possible (see for instance the proof of \eqref{explicitstripes} below) to compute the energy of periodic stripes $E_h$ of period $h$ to get 
\[
   \Fj(E_h)\simeq -\frac{\tau}{h}+ h^{-(p-d)}.
\]
Optimizing in $h$, we find that the optimal stripes have a width of order $\tau^{-1/(p-d-1)}$ and energy of order $-\tau^{(p-d)/(p-d-1)}$. Letting $\beta:= p-d-1$, this motivates the rescaling
\begin{equation}\label{changevar}
   x:=\tau^{-1/\beta}\widehat{x}, \quad L:=\tau^{-1/\beta}\widehat{L} \quad \textrm{and} \quad \Fj(E):= \tau^{(p-d)/\beta}  \mathcal{F}_{\tau,\widehat L}( \widehat{E}).
\end{equation}

In these variables, the optimal stripes have width of order one. From now on, when there is no ambiguity, we will drop the hats. Let us define $\widehat{K}_\tau(z):=\int_{\R^{d-1}} K_\tau(z,\xi) d\xi$. Since  for $z\in \R$,
\begin{equation} 
   \label{eq:1dintegra1}
   \begin{split}
      \frac{1}{|z|^{p-d+1} +\tau^{(p-d+1)/\beta}} \les 
      \int_{\R^{d-1}} \frac{1}{(|z|^{2} + |\xi|^{2})^{p/2} +\tau^{p/\beta}} \d \xi \les
      \frac{1}{|z|^{p-d+1} +\tau^{(p-d+1)/\beta}}, 
   \end{split}
\end{equation} 
we have $\widehat{K}_\tau(z)\simeq \frac{1}{|z|^{q} +\tau^{q/\beta}}$, where we have let $q:=p-d+1$.
We then let for $i\in \{1,\ldots,d\}$,
\[
   \GtL^i(E):=\frac{1}{L^d}\int_{\R} \widehat{K}_\tau(\zeta_i)\lt[\int_{\partial E \cap Q_L} |\nu^E_i| |\zeta_i|-\int_{Q_L} |\chi_E(x)-\chi_E(x+\zeta_i)|\rt],\]
and
\[
   I_{\tau,L}(E):=\frac{2}{d L^d}\int_{Q_L\times \R^d} K_\tau(\zeta) \sum_{i=1}^d|\chi_E(x)-\chi_{E}(x+\zeta_i)||\chi_E(x)-\chi_E(x+\zeta_i^\perp)|.
\]
Notice that \eqref{bandestim} in particular implies that for every $\zeta$ and $i\in\{1,\ldots,d\}$,
\begin{equation}\label{positivG}
   \int_{\partial E \cap Q_L} |\nu^E_i| |\zeta_i|-\int_{Q_L} |\chi_E(x)-\chi_E(x+\zeta_i)|\ge 0.
\end{equation}
We define also  for $E\subset \R$, $L-$periodic and of finite perimeter, the one-dimensional functionals
\[
   \GtL^{1d}(E):=\int_{\R} \widehat{K}_\tau(z) \lt( \per(E,[0,L))|z|-\int_0^L |\chi_E(x)-\chi_E(x+z)|\rt),
\]
so that by Fubini,
\[
   \GtL^i(E)=\frac{1}{L^d}\int_{[0,L]^{d-1}} \GtL^{1d}(E_{x_i^\perp}) dx_i^\perp. 
\]

\begin{lemma}
   For every $Q_L-$periodic set $E$ of finite perimeter, we have 
   \[
      \FtL(E)=\frac{1}{L^d}\lt(-\Pe(E,Q_L)+\int_{\R^d} K_\tau(\zeta) \lt[\int_{\partial E \cap Q_L} \sum_{i=1}^d|\nu^E_i| |\zeta_i|-\int_{Q_L}|\chi_E(x)-\chi_E(x+\zeta)|\rt]\rt).
   \]
   By Lemma \ref{lemsplit}, this yields 
   \[
      \FtL(E)\ge -\frac{1}{L^d} \Pe(E,Q_L)+\sum_{i=1}^d \GtL^i(E)+I_{\tau,L}(E).
   \]

\end{lemma}

\begin{proof}
   By writing that $J\Pe(E,Q_L)=-\tau\Pe(E,Q_L)+\int_{\R^d} K_1(\zeta)\int_{\partial E\cap Q_L} \sum_{i=1}^d |\nu_i^E||\zeta_i|$, we get 
   \[
      \Fj(E)=\frac{1}{L^d}\lt(-\tau \Pe(E,Q_L)+ \int_{\R^d} K_1(\zeta)\lt[\int_{\partial E\cap Q_L} \sum_{i=1}^d |\nu_i^E||\zeta_i| -\int_{Q_L} |\chi_E(x)-\chi_{E}(x+\zeta)|\rt]\rt).
   \]
   Making the change of variables given in \eqref{changevar} and letting also $\zeta:=\tau^{-1/\beta} \widehat{\zeta}$, we obtain
   \begin{align*}
      \Fj(E)=&\frac{\tau^{ d/\beta}}{\widehat{L}^d}\lt(-\tau^{(p-2d)/\beta} \Pe(\widehat{E},Q_{\widehat{L}}) \rt.\\
      &+ \lt.\int_{\R^d} \frac{\tau^{-2d/\beta }}{\tau^{- p/\beta}|\widehat \zeta|^p+1}\lt[\int_{\partial \widehat{E}\cap Q_{\widehat{L}}}\sum_{i=1}^d |\nu_i^{\widehat{E}}||\widehat{\zeta}_i| -\int_{Q_{\widehat{L}}} |\chi_{\widehat{E}}(\widehat{x})-\chi_{\widehat{E}}(\widehat{x}+\widehat{\zeta})|\rt]\rt)\\
      =& \frac{\tau^{ (p-d)/\beta}}{\widehat{L}^d}\lt(- \Pe(\widehat{E},Q_{\widehat{L}})+\int_{\R^d} K_\tau(\widehat \zeta)\lt[\int_{\partial \widehat{E}\cap Q_{\widehat{L}}} \sum_{i=1}^d|\nu_i^{\widehat{E}}||\widehat{\zeta}_i| -\int_{Q_{\widehat{L}}} |\chi_{\widehat{E}}(\widehat{x})-\chi_{\widehat{E}}(\widehat{x}+\widehat{\zeta})|\rt]\rt).
   \end{align*}

\end{proof}

Before closing this section, we prove several estimates which are consequences of  \eqref{toprovebandesim}.

\begin{lemma} 
   \label{lemma:basicComparison}
   Let $E\subset \R$ be a $L-$periodic set of finite perimeter and let $\beta:=p-d-1$. Then, for every $\tau\ge 0$, 

   \begin{equation} 
      \label{eq:basicComparison21}
      \begin{split}
         \GtL^{1d}(E) \gtrsim    \sum_{x\in \partial E\cap [0,L)}  \min(h(x)^{-\beta},\tau^{-1})+\min(g(x)^{-\beta},\tau^{-1}).
      \end{split}
   \end{equation} 
   As a consequence, if $\GtL^{1d}(E)\les \tau^{-1}$, then 
   \begin{equation}\label{estiminf}
      \min_{x\in \partial E\cap [0,L)} \min(h(x),g(x))\ges \GtL^{1d}(E)^{-1/\beta}.
   \end{equation}

   For $L\ge r>0$ and $t\in[0,L)$, let $I_{t}(r):=(t-r/2,t+r/2)$. Then, for every $\delta\ge \tau^{1/\beta}$,
   \begin{equation}\label{estimper}
      \per(E,I_t(r))-1\les  r\delta^{-1}+ \delta^{\beta}\GtL^{1d}(E).
   \end{equation}
In particular, for $r=L$, after optimizing in $\delta$ this implies
\begin{equation}\label{optimdelta}
 L^{-1}\per(E,[0,L))\les L^{-1}+ \max(\tau L^{-1}\GtL^{1d}(E), (L^{-1}\GtL^{1d}(E))^{1/(p-d)}).
\end{equation}

\end{lemma} 

\begin{proof}
   We start by proving \eqref{eq:basicComparison21}.  By \eqref{toprovebandesim},
   \begin{align*}
      \GtL^{1d}(E)&=\int_{\R} \widehat{K}_\tau(z)\lt( |z| \per(E,[0,L))-\int_0^L |\chi_{E}(x)-\chi_{E}(x+z)|\rt)\\
      &\ge \int_{\R} \widehat{K}_\tau(z)\sum_{x\in \partial E\cap [0,L)} \lt( |z|-\eta(x,z)\rt)\\
      &\ges  \sum_{x\in \partial E\cap [0,L)}  \int_{\R} \frac{1}{|z|^{q}+\tau^{ q/\beta}}\lt( |z|-\eta(x,z)\rt),
   \end{align*}
   where in the last line, we have used \eqref{eq:1dintegra1}. 
   Assume now that $E=\cup_{i\in \Z} (s_i,t_i)$ and that $x=s_i$ for some $i\in \Z$ (the case $x=t_i$ can be treated analogously). Then, for $g(x)\ge z\ge -h(x)$, $|z|-\eta(x,z)=0$. Therefore,
   \begin{align*}
      \GtL^{1d}(E)&\ges  \sum_{x\in \partial E\cap [0,L)} \int_{h(x)}^{+\infty} \frac{z-h(x)}{z^{q}+\tau^{ q/\beta}}+\int_{g(x)}^{+\infty} \frac{z-g(x)}{z^{q}+\tau^{ q/\beta}}\\
      &\ges  \sum_{x\in \partial E\cap [0,L)}  \min(h(x)^{-\beta},\tau^{-1})+\min(g(x)^{-\beta},\tau^{-1}),
   \end{align*}
   where we have used that $q-2=\beta$.
   Estimate \eqref{estiminf} follows directly from \eqref{eq:basicComparison21}. Let us finally prove \eqref{estimper}. Let 
   \[A_\delta:=\{x\in \partial E \cap I_t(r) \ : \min(h(x),g(x))\ge \delta\}, \]
   so that in particular $\sharp A_\delta\le 1+ C r\delta^{-1}$ and for $x\in A_\delta^c\cap \partial E \cap I_t(r)$, since $\delta\ge \tau^{1/\beta}$,
   \[
      1\les \delta^{\beta} \lt( \min(h(x)^{-\beta}, \tau^{-1})+ \min(g(x)^{-\beta}, \tau^{-1})\rt).
   \]

   Using \eqref{eq:basicComparison21}, we obtain 
   \begin{align*}
      \per(E,I_t(r))-1&=-1+\sum_{x\in A_\delta} 1 + \sum_{x\in A_\delta^c\cap \partial E\cap I_t(r) } 1\\
      &\les r\delta^{-1} +  \sum_{x\in A_\delta^c\cap \partial E\cap I_t(r) } \delta^{\beta} \lt( \min(h(x)^{-\beta}, \tau^{-1})+ \min(g(x)^{-\beta}, \tau^{-1})\rt)\\
      &\les r\delta^{-1}+ \delta^{\beta} \GtL^{1d}(E),
   \end{align*}
   which proves \eqref{estimper}. 

\end{proof}

\begin{remark}
 Estimate \eqref{optimdelta} shows that for every $i$, the function $x_i^\perp\to \per(E_{x_i^\perp},[0,L))$ is almost in $L^{p-d}([0,L)^{d-1})$.
\end{remark}

A simple consequence of \eqref{estimper} is that the perimeter and the energy are controlled by the non-local terms.
\begin{lemma}
  For every $L\ges 1$ and $\tau \les 1$, and every $Q_L-$periodic set $E$, 
   \begin{equation}\label{estimener}
      \FtL(E)\ges -1+ \sum_{i=1}^d\GtL^i(E)+I_{\tau,L}(E)
   \end{equation}
   and 
   \begin{equation}\label{estimperglob}
      \Pe(E,Q_L)\les L^d \max(1,\FtL(E)).
   \end{equation}
\end{lemma}
\begin{proof}
   Estimate \eqref{estimener} follows from integrating \eqref{optimdelta} and Young's inequality.  Estimate \eqref{estimperglob} then follows  by integrating \eqref{estimper} applied to $r=L$ and $\delta=1$. 
\end{proof}
\begin{remark}\label{rem:scaling}
 For $L\ges 1$ and $\tau\les 1$, since $\GtL^i(E)$ and $I_{\tau,L}$ are non-negative, \eqref{estimener} yields the uniform lower bound
 \[
  \min_{E} \FtL(E)\ges -1.
 \]
The corresponding upper bound can be readily obtained by computing the energy of periodic stripes with width of order one (see Section \ref{sec:1d} for instance).
\end{remark}

We finally give another  consequence of  \eqref{estimper} which resembles \eqref{optimdelta}. Since we are going to use it only for $\tau=0$, we state it only in that case but an analogous statement holds for $\tau>0$.
\begin{lemma}
   For every $Q_L-$periodic set $E$ of finite perimeter and every  $m\in \N$ with $m\ge 2$, $t\in [0,L)$, $L\ge r>0$ and $i\in\{1,\ldots,d\}$,
   \begin{equation}\label{estimbigbadset}
      |\{x_i^\perp\in[0,L)^{d-1} \ : \ \per(E_{x_i^\perp},I_t(r))=m\}|\les  \GzL^i(E) L^{d}r^{\beta} \lt(\frac{1}{m-1}\rt)^{p-d}.  
   \end{equation}

\end{lemma}
\begin{proof}
   Let $\mathcal{B}_m(I_t(r)):=\{x_i^\perp\in [0,L)^{d-1} \ : \ \per(E_{x_i^\perp},I_t(r))=m\}$ and   let  $x_i^\perp\in \mathcal{B}_m(I_t(r))$. Thanks to \eqref{estimper}, for every $\delta>0$,
   \[
      m-1\les r\delta^{-1}+\delta^{\beta} \GzL^{1d}(E_{x_i^\perp}).
   \]
   Optimizing in $\delta$, we get $m-1\les r^{\beta/(\beta +1)}\GzL^{1d}(E_{x_i^\perp})^{1/(\beta+1)}$, which can be equivalently written as 
   \[
      \GzL^{1d}(E_{x_i^\perp})\ges r^{-\beta} (m-1)^{p-d}.
   \]
   We can thus conclude that 
   \[
      \GzL^i(E)\ge \frac{1}{L^d}\int_{\mathcal{B}_m(I_t(r))} \GzL^{1d}(E_{x_i^\perp})\ges |\mathcal{B}_m(I_t(r))|L^{-d} r^{-\beta} (m-1)^{p-d},
   \]
   from which \eqref{estimbigbadset} follows.\\
\end{proof}

\section{A rigidity result}\label{sec:rigidity}
In this section, we prove that in the limit $\tau=0$, sets of finite energy must be stripes. For a set $E$ of finite perimeter, we introduce the measures
\[
   \mu_i:=|\partial \chi_{E_{x_i^\perp}}| \otimes dx_i^\perp,
\]
so that actually  by the slicing formula \cite[Th. 3.108]{AFPBV} (see also \cite[Cor. 2.29]{AFPBV}), $|\partial_i \chi_E|=\mu_i$ for $i\in\{1,\ldots,d\}$.

We then define for $x\in Q_L$ and $i\in\{1,\ldots,d\}$ the ``cubic'' upper $(d-1)-$dimensional densities 
\[\Theta_i(x):=\limsup_{r\to 0} \frac{\mu_i(Q_x(r))}{r^{d-1}},\]
where $Q_x(r):=x+[-r/2,r/2)^d$. We recall that the classical upper $(d-1)-$dimensional densities are defined by \cite[Def. 2.55]{AFPBV}
\[
\Theta^*_i(x):=\limsup_{r\to 0} \frac{\mu_i(B_r(x))}{\omega_{d-1}r^{d-1}},
\]
where $B_r(x)$ is the ball of radius $r$ centered in $x$ and where $\omega_{d-1}$ is the volume of the unit ball of $\R^{d-1}$. Notice that of course for every $x\in Q_L$,
\begin{equation}\label{equivTheta}
 \Theta_i(x)\sim \Theta_i^*(x).
\end{equation}

\begin{lemma}\label{lemTheta}
   Let $E$ be a $Q_L-$ periodic set of finite perimeter and such that $\sum_{i=1}^d \GzL^i(E)+I_{0,L}(E)<+\infty$. Then, for every $x\in Q_L$ and every $i\in\{1,\ldots,d\}$, $\Theta_i(x)\in\{0,1\}$.
\end{lemma}
\begin{proof}
   For definiteness, we prove the assertion for $\Theta_1$. Let us first show that $\Theta_1\le 1$. We recall that $\beta=p-d-1\ge d-1$.

   For $r>0$, let 
   \[\mathcal{S}_r:=\{x_1^\perp \in[0,L)^{d-1} \ : \ \min_{x_1\in \partial E_{x_1^\perp}} \min( g_{x_1^\perp}(x_1), h_{x_1^\perp}(x_1))>r\}.\] 
   For $x_1^\perp\in \mathcal{S}_r^c$, by \eqref{estiminf}, $r\ges \GzL^{1d}(E_{x_1^\perp})^{-1/\beta}$, that is $\GzL^{1d}(E_{x_1^\perp})\ges r^{-\beta}$. Integrating this, we get that 
   \begin{equation}\label{estimSbad}
      |\mathcal{S}_r^c|\les  \GzL^1(E) L^d r^{\beta}. 
   \end{equation}

   We claim that for $\bar{x}=(\bar x_1, \bar{x}_1^\perp)\in Q_L$,
   \begin{equation}\label{thetaonegood}
      \Theta_1(\bar x)=\limsup_{r\to 0} \frac{1}{r^{d-1}} \int_{Q'_{\bar x_1^\perp}(r)\cap \mathcal{S}_r}\per(E_{x_1^\perp},I_{\bar x_1}(r)) dx_1^\perp,
   \end{equation}
   where  for $(x_1,x_1^\perp)\in [0,L)^d$ and $r>0$, $Q'_{\bar x_1^\perp}(r):=\bar x_1^\perp+[-r/2,r/2)^{d-1}$ and $I_{\bar x_1}(r):= \bar x_1+[-r/2,r/2)$.
   Indeed, letting as in the proof of \eqref{estimbigbadset}, $\mathcal{B}_m(I_{\bar x_1}(r)):=\{x_1^\perp \in [0,L)^{d-1} \ : \ \per(E_{x_1^\perp}, I_{\bar x_1}(r))=m\}$, for $r>0$ we have
   \begin{multline*}
      \int_{Q'_{\bar x_1^\perp}(r)\cap \mathcal{S}^c_r}\per(E_{x_1^\perp},I_{\bar x_1}(r)) dx_1^\perp=   \int_{Q'_{\bar x_1^\perp}(r)\cap \mathcal{S}^c_r\cap \mathcal{B}^c_2(I_{\bar x_1}(r))}\per(E_{x_1^\perp},I_{\bar x_1}(r)) dx_1^\perp\\
      + \int_{Q'_{\bar x_1^\perp}(r)\cap \mathcal{S}^c_r\cap\mathcal{B}_2(I_{\bar x_1}(r))}\per(E_{x_1^\perp},I_{\bar x_1}(r)) dx_1^\perp.
   \end{multline*}
   Since on the one hand, from \eqref{estimSbad}
   \[
      \int_{Q'_{\bar x_1^\perp}(r)\cap \mathcal{S}^c_r\cap \mathcal{B}_2(I_{\bar x_1}(r))}\per(E_{x_1^\perp},I_{\bar x_1}(r)) dx_1^\perp\le 2|\mathcal{S}_r^c|\les  \GzL^1(E) L^d r^{\beta},
   \]
   and on the other hand, thanks to \eqref{estimbigbadset} and $\beta>d-1\ge 1$,
   \[
      \int_{Q'_{\bar x_1^\perp}(r)\cap \mathcal{S}^c_r\cap\mathcal{B}_2^c(I_{\bar x_1}(r))}\per(E_{x_1^\perp},I_{\bar x_1}(r)) dx_1^\perp\le \sum_{m=3}^{+\infty} m|\mathcal{B}_m(I_{\bar x_1}(r))|\les  \GzL^1(E) L^{d}r^{\beta}, 
   \]
   we get 
   \begin{equation*}
      \int_{Q'_{\bar x_1^\perp}(r)\cap \mathcal{S}^c_r}\per(E_{x_1^\perp},I_{\bar x_1}(r)) dx_1^\perp\les  \GzL^1(E) L^d r^{\beta}.
   \end{equation*}
   From this \eqref{thetaonegood} follows. Since for  $x_1^\perp\in \mathcal{S}_r$ and $I\subset [0,L)$ with $|I|\le r$, $\per(E_{x_1^\perp},I)\in\{0,1\}$, \eqref{thetaonegood} implies that $\Theta_1\le 1$.\\

   Assume  now  for the sake of contradiction that there exists $\bar x\in Q_L$ such that $0<\Theta_1(\bar x)<1$. Let
   \begin{multline*}
      A:=\{x_1^\perp\in Q'_{\bar x_1^\perp}(r)\cap \mathcal{S}_r \ : \ \per(E_{x_1^\perp},I_{\bar x_1}(r))=1 \} \qquad \textrm{and} \\ B:=\{x_1^\perp\in Q'_{\bar x_1^\perp}(r)\cap \mathcal{S}_r \ : \ \per(E_{x_1^\perp},I_{\bar x_1}(r))=0 \}.
   \end{multline*}
   Then, thanks to \eqref{thetaonegood}, there exists $\delta>0$ such that for all  $\bar r>0$ there exists $0<r\le \bar r$, with
   \[ \delta r^{d-1} \le |A| \le (1-\delta)r^{d-1} \qquad \textrm{and} \qquad \delta r^{d-1} \le |B|\le (1-\delta)r^{d-1}.\]
   Letting 
   \[\widetilde{A}:=\{x_1^\perp\in A \ : E_{x_1^\perp }\cap I_{\bar x_1}(r)=(s(x_1^\perp), \bar x_1+r/2)\}\quad  \textrm{and} \quad \widetilde{B}:=\{x_1^\perp\in B : \ E_{x_1^\perp}\cap I_{\bar x_1}(r)=I_{\bar x_1}(r)\},\]
   we may assume without loss of generality that 
   \begin{equation}\label{hypAB} \delta r^{d-1}/2 \le |\widetilde{A}| \le (1-\delta)r^{d-1} \qquad \textrm{and} \qquad \delta r^{d-1}/2 \le |\widetilde{B}|\le (1-\delta)r^{d-1}.\end{equation}
   Indeed, the case when \eqref{hypAB} holds with $\widetilde{A}^c$ (respectively $\widetilde{B}^c$) instead of $\widetilde{A}$ (respectively $\widetilde{B}$) can be similarly treated. Since $(-r/2,r/2)^{d-1}\subset B_{\frac{1}{2}\sqrt{d-1}r}(0)$, for every $x_1^\perp\in Q'_{\bar x_1^\perp}(r)$, 
   \[
    Q'_{\bar x_1^\perp}(r)\subset x_1^\perp+B_{\frac{3}{2}\sqrt{d-1}r}(0)
   \]
so that reducing the integral defining $I_{0,L}(E)$ to the set
   \[\{x_1^\perp\in \widetilde{A}, \ x_1\in (s(x_1^\perp)-r,s(x_1^\perp))\subset E_{x_1^\perp}^c, \  x_1+\zeta_1\in (s(x_1^\perp),s(x_1^\perp)+r)\subset E_{x_1^\perp} \textrm{ and } x_1^\perp+\zeta_1^\perp\in \widetilde{B}\},\]
   we may now estimate
   \begin{align*}
      I_{0,L}(E)&\ge \frac{2}{d L^d}\int_{\widetilde{A}}\int_{s(x_1^\perp)-r}^{s(x_1^\perp)}\int_{s(x_1^\perp)-x_1}^{s(x_1^\perp)-x_1+r}\int_{|\zeta_1^\perp|\le \frac{3}{2}\sqrt{d-1}r}\frac{\chi_{\widetilde{B}}(x_1^\perp+\zeta_1^\perp)}{|\zeta|^p}d \zeta_1^\perp d\zeta_1 d x_1 dx_1^\perp\\
      &\ges\frac{1}{L^d r^p}\int_{\widetilde{A}}\int_{s(x_1^\perp)-r}^{s(x_1^\perp)}\int_{s(x_1^\perp)-x_1}^{s(x_1^\perp)-x_1+r}\int_{Q'_{\bar{x}_1^\perp}(r)}\chi_{\widetilde{B}}(\zeta_1^\perp) d \zeta_1^\perp d\zeta_1 d x_1 dx_1^\perp\\
      &\ges\frac{1}{L^d r^p}  \int_{\widetilde{A}}\int_{s(x_1^\perp)-r}^{s(x_1^\perp)}\int_{s(x_1^\perp)-x_1}^{s(x_1^\perp)-x_1+r} |\widetilde{B}|d\zeta_1 d x_1 dx_1^\perp\\
      &\ges \frac{\delta^2}{L^d r^{p-2d}}, 
   \end{align*}
   which using that $r^{2d-p}$ is unbounded as $r$ goes to zero, contradicts the fact that $I_{0,L}(E)$ is finite.
\end{proof}

\begin{lemma}\label{lemline}
   Let $E$ be a $Q_L-$periodic set of finite perimeter such that $\sum_{i=1}^d\GzL^i(E)+I_{0,L}(E)<+\infty$. For $i\in \{1,\ldots,d\}$, if $\bar x\in Q_L$ is such that $\Theta_i(\bar x)=1$ then $\Theta_i(\bar x+\zeta_i^\perp)=1$ for all $\zeta\in \R^d$. 
\end{lemma}
\begin{proof}
   Assume for definiteness that $\Theta_1(\bar x)=1$ and let $\bar\zeta_1^\perp\in \R^{d-1}$ be such that $\Theta_1(\bar x+\bar \zeta_1^\perp)=0$. For all $\bar r>0$, there exists $r<\bar r$ such that 
   \[
      \mu_1(Q_{\bar x}(r))\ge \frac{3}{4}r^{d-1} \qquad \textrm{and} \qquad \mu_1(Q_{\bar x+\bar \zeta_1^\perp}(r))\le\frac{1}{4} r^{d-1}. 
   \]
   Since $\zeta_1^\perp\to \mu_1(Q_{\bar x+\zeta_1^\perp}(r))$ is continuous (this is a consequence of $x_1^\perp\to \per(E_{x_1^\perp},I_{\bar x_1}(r))\in L^1((0,L)^{d-1})$), there exists $\zeta_1^\perp\in (0,\bar \zeta_1^\perp)$ such that 
   \[
      \mu_1(Q_{\bar x+ \zeta_1^\perp}(r))=\frac{1}{2} r^{d-1}.
   \]
   Arguing exactly as in the proof of Lemma \ref{lemTheta}, we reach a contradiction.
\end{proof}

We are finally in position to prove a rigidity result for sets of finite energy.
\begin{proposition}\label{rigidity}
   Let $E$ be a $Q_L-$periodic set of finite perimeter such that $\sum_{i=1}^d\GzL^i(E)+I_{0,L}(E)<+\infty$. Then, $E$ is one-dimensional i.e. up to permutation of the coordinates, $E=\widehat{E}\times \R^{d-1}$ for some $L-$periodic set $\widehat{E}$.
\end{proposition}
\begin{proof} 
   By Lemma \ref{lemline}, if $\bar x$ is such that $\Theta_i(\bar x)=1$ for some $i\in\{1,\ldots,d\}$, then for every $\zeta_i^\perp$, $\Theta_i(\bar x+\zeta_i^\perp)=1$, which in turn by \eqref{equivTheta} and \cite[Thm. 2.56]{AFPBV} implies that $\bar x +\zeta_i^\perp\subset \partial E$ . Since $E$ has finite perimeter (in $Q_L$), 
   it may contain at most a finite number of such hyperplanes.   If now $Q$ is a cube which does not intersect any of these hyperplanes, then by \eqref{equivTheta}, $\Theta_i^*=\Theta_i=0$ in $Q$ for every $i\in\{1,\ldots,d\}$ and therefore by \cite[Thm. 2.56]{AFPBV} again, $|D \chi_E|(Q)=0$ so that either $Q\subset E$ or $Q\subset E^c$. 
   In $Q_L$, the set $E$ is thus made of a finite union of hyperrectangles, 
   which constitute a checkerboard structure. Arguing as in the last part of the proof of Lemma \ref{lemTheta}, we obtain that $I_{0,L}(E)=+\infty$ unless this checkerboard is one-dimensional.
\end{proof}

\begin{remark}
   For a set $E$ of finite perimeter it can be readily  seen that  $I_{0,L}(E)$ finite, implies that every blow-up of $E$ is an hyperplane orthogonal to some coordinate axis. This in particular implies that  for $\mathcal{H}^{d-1}$-a.e. $x\in \partial E$, $\nu^E=e_i$ for some $i\in \{1,\ldots,d\}$ (with $i$ depending on $x$).
   However, it does not seems to be easy to conclude from this fact that $E$ is one dimensional. Indeed, sets of finite perimeter can be very badly behaved (see \cite[Ex. 3.53]{AFPBV} for instance). In a work in progress, we will show that the conclusion of Proposition \ref{rigidity} actually holds without assuming that $E$
   is of finite perimeter or that $\GzL^i(E)$ is finite.   
\end{remark}

\section{The Gamma-convergence result}\label{sec:gamma} 

In this section, we prove our main result, which is the $\Gamma-$convergence of $\FtL$ to $\FzL$. Recall that 
\[
   \FzL(E):=\begin{cases}
      \frac{1}{L}\lt( -\per(\widehat{E},[0,L))+ \GzL^{1d}(\widehat{E}) \rt) & \textrm{if } E=\widehat{E}\times \R^{d-1}  \textrm{ for some $L-$periodic}\\
      & \textrm{ set } \widehat{E} \textrm{ of finite perimeter},\\
      +\infty & \textrm{otherwise.}
   \end{cases}
\]

\begin{theorem}\label{prop:GammaLimit}
   It holds:
   \begin{itemize}
      \item[i)][Compactness and lower bound] Let $E^\tau$ be a sequence of $Q_L-$periodic sets such that $\sup_\tau \FtL(E^\tau)<+\infty$, then up to a subsequence and a relabeling of the coordinate axes, 
      $E^\tau$ converges strongly in  $L^1$ to some one-dimensional $Q_L-$periodic set $E$ of finite perimeter. Moreover,
         \begin{equation}\label{gammaliminf}
            \liminf_{\tau\to 0} \FtL(E^\tau)\ge \FzL(E).
         \end{equation}
      \item[ii)][Upper bound] For every set $E$ with $\FzL(E)<+\infty$, there exists a sequence $E^\tau\to E$ with 
         \begin{equation}\label{gammalimsup}
            \limsup_{\tau\to 0} \FtL(E^\tau)\le \FzL(E).
         \end{equation}

   \end{itemize}
\end{theorem}
\begin{proof}
   We start by proving i). Let $E^\tau$ be such that $\sup_\tau \FtL(E^\tau)<+\infty$. Then, by \eqref{estimperglob}, $\sup_\tau \Pe(E^\tau,Q_L)<+\infty$ so that we may extract a subsequence converging in $L^1$ to some $Q_L-$periodic set $E$ of finite perimeter. Let us first prove that 
   \begin{equation}\label{lowerNL}
      \liminf_{\tau\to 0} \GtL^i(E^\tau)\ge \GzL^i(E)\  \textrm{ for } i\in\{1,\ldots,d\} \qquad \textrm{and}\qquad \liminf_{\tau\to 0} I_{\tau,L}(E^\tau)\ge I_{0,L}(E).
   \end{equation}
   By \eqref{estimener} and Proposition \ref{rigidity}, this would prove that $E$ is one-dimensional. For definiteness, let us prove the inequality concerning $\GtL^1$. The proof of the related lower bound for $I_{\tau,L}$ is similar (and actually simpler). For $\tau>\tau'$, since $\widehat{K}_\tau\le \widehat{K}_{\tau'}$ and recalling \eqref{positivG},
   \[
      \GtL^1(E^{\tau'})\le \mathcal{G}^1_{\tau',L}(E^{\tau'}).
   \]
   Now, if $\tau$ is fixed, by Fatou and \eqref{positivG},
   \begin{align*}\liminf_{\tau'\to 0} \GtL^1(E^{\tau'})&\ge\frac{1}{L^d} \int_{\R} \widehat{K}_{\tau}(\zeta_1)\liminf_{\tau'\to 0} \lt[ \int_{\partial E^{\tau'}\cap Q_L} |\nu^{E^{\tau'}}_1||\zeta_1|-\int_{Q_L} |\chi_{E^{\tau'}}(x)-\chi_{E^{\tau'}}(x+\zeta_1)|\rt] \\
      &\ge \frac{1}{L^d}\int_{\R} \widehat{K}_{\tau}(\zeta_1) \lt[ \int_{\partial E\cap Q_L} |\nu^{E}_1||\zeta_1|-\int_{Q_L} |\chi_{E}(x)-\chi_{E}(x+\zeta_1)|\rt]\\
      &=\GtL^1(E),
   \end{align*}
   where we have used that for fixed $\zeta_1$, $\int_{\partial E\cap Q_L} |\nu^{E}_1||\zeta_1|-\int_{Q_L} |\chi_{E}(x)-\chi_{E}(x+\zeta_1)|$ is lower semicontinuous with respect to $L^1$ convergence. Finally, using again \eqref{positivG} and the monotone convergence theorem, we have 
   \[
      \liminf_{\tau'\to 0}\mathcal{G}^1_{\tau',L}(E^{\tau'})\ge \lim_{\tau\to 0} \GtL^1(E)=\GzL^1(E),
   \]
   which proves \eqref{lowerNL}. From this point, in order to show the lower bound \eqref{gammaliminf}, we are left to check that 
   \[
      \liminf_{\tau \to 0} -\Pe(E^\tau,Q_L)\ge -\Pe(E,Q_L).
   \]
  This is not straightforward since we have a minus sign in front of the perimeter. By slicing it is  enough to prove that for $i\in\{1,\ldots,d\}$,
   \begin{equation}\label{lowerper}
      \liminf_{\tau \to 0} -\int_{[0,L)^{d-1}}\per(E^\tau_{x_i^\perp},[0,L))\ge -\int_{[0,L)^{d-1}}\per(E_{x_i^\perp},[0,L)).
   \end{equation}
   To ease a bit the notation, we write $f_\tau(x_i^\perp):=\per(E^\tau_{x_i^\perp},[0,L))$ and $f(x_i^\perp):=\per(E_{x_i^\perp},[0,L))$. Let now $\delta>0$ be fixed. By \eqref{estiminf}, if $\GtL^{1d}(E_{x_i^\perp}^\tau)\les \delta^{-\beta}$ then 
   \[
    \min_{x_i\in \partial E^\tau_{x_i^\perp}}\min(h_{x_i^\perp}^\tau(x_i),g_{x_i^\perp}^\tau(x_i))\ge \delta.
   \]
Therefore if also $|E_{x_i^\perp}^\tau\Delta E_{x_i^\perp}|\le \delta$, then   $f_\tau(x_i^\perp)=f(x_i^\perp)$. We can thus compute
\begin{align*}
 \liminf_{\tau \to 0} -\int_{[0,L)^{d-1}} f_\tau&\ge \liminf_{\tau \to 0} -\int_{\{\GtL^{1d}(E_{x_i^\perp}^\tau)\les \delta^{-\beta}\} \cap\{|E_{x_i^\perp}^\tau\Delta E_{x_i^\perp}|\le \delta\} } f_\tau +\liminf_{\tau \to 0} -\int_{\{|E_{x_i^\perp}^\tau\Delta E_{x_i^\perp}|\ge \delta\} } f_\tau\\
 &\qquad + \liminf_{\tau \to 0} -\int_{\{\GtL^{1d}(E_{x_i^\perp}^\tau)\ges \delta^{-\beta}\} } f_\tau\\
 &=\liminf_{\tau \to 0} -\int_{\{\GtL^{1d}(E_{x_i^\perp}^\tau)\les \delta^{-\beta}\} \cap\{|E_{x_i^\perp}^\tau\Delta E_{x_i^\perp}|\le \delta\} } f -\limsup_{\tau \to 0} \int_{\{|E_{x_i^\perp}^\tau\Delta E_{x_i^\perp}|\ge \delta\} } f_\tau\\
 &\qquad -\limsup_{\tau \to 0} \int_{\{\GtL^{1d}(E_{x_i^\perp}^\tau)\ges \delta^{-\beta}\} } f_\tau\\
&\ge -\int_{[0,L)^{d-1} } f -\limsup_{\tau \to 0} \int_{\{|E_{x_i^\perp}^\tau\Delta E_{x_i^\perp}|\ge \delta\} } f_\tau -\limsup_{\tau \to 0} \int_{\{\GtL^{1d}(E_{x_i^\perp}^\tau)\ges \delta^{-\beta}\} } f_\tau.
 \end{align*}
By \eqref{optimdelta} and H\"older inequality, 
   \begin{align*}
    \limsup_{\tau \to 0} \int_{\{|E_{x_i^\perp}^\tau\Delta E_{x_i^\perp}|\ge \delta\} } f_\tau&\les  \limsup_{\tau \to 0} \lt[ |\{|E_{x_i^\perp}^\tau\Delta E_{x_i^\perp}|\ge \delta\}|+\tau \int_{[0,L)^{d-1}} \GtL^{1d}(E^\tau_{x_i^\perp})\rt.\\
    &\qquad  \qquad \lt.+L \int_{\{|E_{x_i^\perp}^\tau\Delta E_{x_i^\perp}|\ge \delta\}} (L^{-1} \GtL^{1d}(E^\tau_{x_i^\perp}))^{1/(p-d)}\rt]\\
    &\les\limsup_{\tau \to 0} \Bigg[|\{|E_{x_i^\perp}^\tau\Delta E_{x_i^\perp}|\ge \delta\}|\\
    &\qquad \lt.+ L\lt(\int_{[0,L)^{d-1}} L^{-1}\GtL^{1d}(E^\tau_{x_i^\perp})\rt)^{1/(p-d)}|\{|E_{x_i^\perp}^\tau\Delta E_{x_i^\perp}|\ge \delta\}|^{\frac{p-d-1}{p-d}}\rt]\\
    &=0,
   \end{align*}
where in the last line we used that by Fubini, since $|E^\tau\Delta E|\to 0$, also $|\{|E_{x_i^\perp}^\tau\Delta E_{x_i^\perp}|\ge \delta\}|\to 0$. Analogously, using that 
\[|\{\GtL^{1d}(E_{x_i^\perp}^\tau)\ges \delta^{-\beta}\}|\les \delta^\beta \int_{[0,L)^{d-1}} \GtL^{1d}(E^\tau_{x_i^\perp})\le C \delta^\beta\]
where now $C$ depends on $L$ and on $\sup_\tau \FtL(E^\tau)$, we obtain for $\delta<1$,
\begin{align*}
    \limsup_{\tau \to 0} \int_{\{\GtL^{1d}(E_{x_i^\perp}^\tau)\ges \delta^{-\beta}\} } f_\tau&\les \limsup_{\tau \to 0} \lt[|\{\GtL^{1d}(E_{x_i^\perp}^\tau)\ges \delta^{-\beta}\}| +\tau \int_{[0,L)^{d-1}} \GtL^{1d}(E^\tau_{x_i^\perp})\rt.\\
   &\qquad \lt.+ L\lt(\int_{[0,L)^{d-1}} L^{-1}\GtL^{1d}(E^\tau_{x_i^\perp})\rt)^{1/(p-d)}|\{\GtL^{1d}(E_{x_i^\perp}^\tau)\ges \delta^{-\beta}\}|^{\frac{p-d-1}{p-d}}\rt]\\
   &\le C\lt(\delta^\beta+\delta^{\beta \frac{p-d-1}{p-d}}\rt)\le C\delta^{\beta \frac{p-d-1}{p-d}}, 
   \end{align*}
   where $C$ depends on $L$ and on $\sup_\tau \FtL(E^\tau)$. Putting these together, we get 
  \[  \liminf_{\tau \to 0} -\int_{[0,L)^{d-1}} f_\tau\ge -\int_{[0,L)^{d-1} } f-C\delta^{\beta \frac{p-d-1}{p-d}}.\]
  Letting finally $\delta\to 0$, we get \eqref{lowerper}.\\

   We may now turn to the proof of \eqref{gammalimsup}. Let $E$ be such that $\FzL(E)<+\infty$. Without loss of generality, we may assume that $E=\widehat{E}\times \R$ for some $L-$periodic set $\widehat{E}$ of finite perimeter. Since $\widehat{E}$ is of finite perimeter, we have that\footnote{In fact by \eqref{estiminf}, we can even get a quantitative estimate of $c_0$ in term of the energy.} 
   \[
      c_0:=\min_{x\in \partial \widehat{E}} \min(h(x),g(x))>0.
   \]
   Arguing as in the proof of \eqref{eq:basicComparison21}, we see that 
   \[
      \GzL^{1d}(\widehat{E})=\int_{|z|\ge c_0} \widehat{K}_0(z)\lt( \per(\widehat{E},[0,L))|z|-\int_0^L|\chi_{\widehat{E}}(x)-\chi_{\widehat{E}}(x+z)|\rt].
   \]
   Since $\widehat{K}_0$ is integrable in $\{|z|\ge c_0\}$, by the dominated convergence theorem,
   \[\lim_{\tau\to 0} \GtL^{1d}(\widehat{E})=\GzL^{1d}(\widehat{E}),\]
   so that we can use $E=\widehat{E}\times \R$ itself as a recovery sequence.
\end{proof}

\section{Minimizers of the one-dimensional problem}\label{sec:1d}

In this section, we prove that minimizers of $\FzL$ are periodic stripes of period essentially not depending on $L$. For a set $E$ with $\FzL(E)<+\infty$, we identify 
by a slight abuse of notation, the set $E$ and corresponding one-dimensional set $\widehat{E}$. That is for a $L-$periodic set $E$ of finite perimeter, we consider 
\begin{equation}\label{energi1D}
   \FzL(E)=\frac{1}{L}\lt(-\per(E,[0,L))+ C_q \int_{\R}\frac{1}{|z|^{q}}\lt[\per(E,[0,L))|z|-\int_0^L |\chi_{E}(x)-\chi_{E}(x+z)|\rt]\rt), 
\end{equation}
where $q:=p-d+1> d+1$ and where  we have used that for $z\in \R$, $\widehat{K}_0(z)=C_q |z|^{-q}$ for some constant $C_q>0$.  Since $E$ is of finite perimeter, we can write  it as  $E=\cup_{i\in \Z} (s_i,t_i)$. As usually, we may assume that $E\cap[0,L)=\cup_{i=1}^N(s_i,t_i)$ for some $N\in \N$,   $s_1>0$ and $t_N<L$. 

For $h>0$, let $E_h:=\cup_{k\in \Z} [(2k)h,(2k+1)h]$. Then, we define
\[e_\infty(h):=\F_{0,2h}(E_h)=\lim_{L\to+\infty} \FzL(E_h).\]
We can now compute
\begin{lemma}\label{lemmahstar}
   Letting 
   \[\bCp:= \frac{4C_q(1-2^{-(q-3)})}{(q-2)(q-1)}\sum_{k\ge 1} \frac{1}{k^{q-2}},\]
   for every $h>0$, it holds
   \begin{equation}\label{explicitstripes}
      e_\infty(h)=-\frac{1}{h}+\bCp  h^{-(q-1)}.
   \end{equation}
   Therefore, $h^\star:=((q-1)\bCp)^{-1/(q-2)}$  is the unique (positive) minimizer of $e_\infty(h)$.
\end{lemma}
\begin{proof}
   Since the contribution of the  perimeter to the energy is clear, we just need to compute the non-local interaction.  
   Denote by
   \begin{equation*} 
      \begin{split}
         A:= \int_{\R}\frac{1}{|z|^{q}}\lt[\per(E_h, [0,2h))|z|-\int_0^{2h} |\chi_{E_h}(x)-\chi_{E_h}(x+z)|\rt].
      \end{split}
   \end{equation*} 
   We start by noting that  
   \begin{align*}
      A=&\int_{\R}\frac{1}{|z|^{q}} \lt(|z|-\int_0^h \chi_{E_h^c}(x+z)\rt) + \int_{\R}\frac{1}{|z|^{q}} \lt(|z|-\int_h^{2h} \chi_{E_h}(x+z)\rt)\\
      =& 2 \int_{\R}\frac{1}{|z|^{q}} \lt(|z|-\int_0^h \chi_{E_h^c}(x+z)\rt) =4 \int_{\R^+} \frac{1}{z^{q}}\lt(z-\int_0^h \chi_{E_h^c}(x+z)\rt),
   \end{align*}
   where we have first made the change of variables $x=y+h$ and used that $x+z\in E_h$ is equivalent to $y+z\in E_h^c$ and then,  for $z<0$, we have  let  $z'=-s$ and $x'=h-x$ (so that if $x+z\in E_h^c$, also $x'+z'\in E_h^c$).
   Hence, we want to show that 
   \begin{equation}\label{toprovestripes}
      \int_{\R^+} \frac{1}{z^{q}}\lt(z-\int_0^h \chi_{E_h^c}(x+z)\rt)=\frac{2(1-2^{-(q-3)})}{(q-2)(q-1)} \sum_{k\ge 1} \frac{1}{k^{q-2}} h^{-(q-2)}.
   \end{equation}
   Since  in $\R^+$, $\chi_{E_h^c}=\chi_{[0,h]^c}-\sum_{k\ge 1} \chi_{[(2k)h,(2k+1)h]}$, 
   \begin{multline*}
      \int_{\R^+} z^{-q} \lt(z-\int_{0}^{h} \chi_{E_h^c}(x+z)\rt)=\int_{\R^+} z^{-q} \lt(z-\int_{0}^{h} \chi_{[0,h]^c}(x+z)\rt)\\
      +\sum_{k\ge 1}\int_{\R^+} z^{-q}\int_0^h \chi_{[(2k)h,(2k+1)h]}(x+z).
   \end{multline*}
   The first term on the right-hand side can be computed as
   \[
      \int_{\R^+} z^{-q} \lt(z-\int_{0}^{h} \chi_{[0,h]^c}(x+z)\rt)=\int_h^{+\infty} z^{-q}(z-h)=\frac{h^{-(q-2)}}{(q-2)(q-1)},\]
   while for the second term we can use that  for $k\ge 1$,
   \begin{align*}
      \int_{\R^+} z^{-q}\int_0^h \chi_{[(2k)h,(2k+1)h]}(x+z)&=\int_{(2k)h}^{(2k+1)h}\int_0^h \frac{dx dy}{(y-x)^{q}}\\
      &=\frac{h^{-(q-2)}}{(q-2)(q-1)} \lt(\frac{1}{(2k-1)^{q-2}}+\frac{1}{(2k+1)^{q-2}}-\frac{2}{(2k)^{q-2}}\rt).
   \end{align*}
   Putting this together, we get 
   \[
      \int_{\R^+} \frac{1}{z^{q}}\lt(z-\int_0^h \chi_{E_h^c}(x+z)\rt)=\frac{h^{-(q-2)}}{(q-2)(q-1)}\lt(1+ \sum_{k\ge 1} \frac{1}{(2k-1)^{q-2}}+\frac{1}{(2k+1)^{q-2}}-\frac{2}{(2k)^{q-2}}\rt).
   \]

   Since
   \begin{equation}\label{sum}
      2(1-2^{-(q-3)})\sum_{k\ge 1} \frac{1}{k^{q-2}}=1+\sum_{k\ge 1} \frac{1}{(2k-1)^{q-2}}+\frac{1}{(2k+1)^{q-2}}-\frac{2}{(2k)^{q-2}},
   \end{equation}
   this concludes the proof of \eqref{toprovestripes}.
\end{proof}
\begin{remark}
   The sum $\sum_{k\ge 1} k^{-(q-2)}$ is finite since $q>d+1>3$. Notice that actually, the sum in the right-hand side of \eqref{sum} is finite for every $q>1$. 
   Therefore, the energy of periodic stripes is finite for every $q>1$ i.e. for $p>d$.
\end{remark}

The main estimate of this section is the following chessboard estimate:
\begin{lemma}\label{chessboardlem}
   For every $L-$periodic set $E$ of finite perimeter, it holds
   \begin{equation}\label{chessboardineq}
      \FzL(E)\ge \frac{1}{2L} \sum_{x\in \partial E\cap [0,L)} h(x) e_\infty(h(x))+g(x) e_{\infty}(g(x)).
   \end{equation}

\end{lemma}
The proof of Lemma~\ref{chessboardlem} will occupy the rest of this section. Before turning to its proof, let us state its main consequence
\begin{theorem}\label{theo:1d}
   For every $L>0$, the minimizers of $\FzL$ are periodic stripes $E_h$ for some $h>0$ satisfying 
   \begin{equation}\label{deltah}
      |h-h^\star|\les \frac{1}{L}.
   \end{equation}
   Moreover, for $L\in 2h^\star \N$, $E_{h^\star}$ is the unique minimizer.
\end{theorem}
\begin{proof}
   Let us first prove the last claim. Let $L>0$ and let  $E$ be any $L$ periodic set then by \eqref{chessboardineq}, and the minimality of $h^\star$ for $e_\infty$,
   \begin{equation}\label{firstestimchess}\FzL(E)\ge  \frac{1}{2L} \sum_{x\in \partial E\cap [0,L)} h(x) e_\infty(h(x))+g(x) e_{\infty}(g(x))\ge \frac{e_\infty(h^\star)}{2L} \sum_{x\in \partial E\cap[0,L)} h(x)+g(x)=e_\infty(h^\star).\end{equation}
   For $L\in 2h^\star \N$, $E_{h^\star}$ is admissible thus we have equalities in \eqref{firstestimchess} for minimizers of $\FzL$. Since $h^\star$
   is the unique minimizer of $e_\infty(h)$, this implies that $h(x)=g(x)=h^\star$ for every $x\in \partial E$, proving the claim.\\

   If now $L$ is arbitrary, using the first inequality in \eqref{firstestimchess} and \eqref{explicitstripes}, we get 
   \[
      L \FzL(E)\ge -\per(E,[0,L))+\frac{1}{2}\bCp \sum_{x\in \partial E\cap [0,L)} h(x)^{-(q-2)}+g(x)^{-(q-2)}.
   \]
   If $\per(E,[0,L))=2N$ is fixed, then, from the supperadditivity of  $x^{-(q-2)}$, 
   \[\min_{\sum_{i=1}^{2N} h_i+g_i=2L} \sum_{i=1}^{2N} h_i^{-(q-2)}+g_i^{-(q-2)}= 4N \lt(\frac{L}{2N}\rt)^{-(q-2)},\]
   and the minimum is attained only at $h_i=g_i=\frac{L}{2N}$. Since $E_{\frac{L}{2N}}$ is admissible and satisfies 
   \[L\FzL(E_{\frac{L}{2N}})=-2N+2\bCp N \lt(\frac{L}{2N}\rt)^{-(q-2)},\]
   we obtain as before that every minimizer has to be equal to $E_{\frac{L}{2N}}$ for some $N\in \N$. Letting $2N^\star:=\frac{L}{h^\star}$, we see that the function $x\to -x+\bCp L^{-(q-2)}x^{q-1}$ is minimized at $x=2N^\star$.
   This implies that letting $h^+:= L (2\lceil L/(2h^\star)\rceil)^{-1}$ and $h_-:=L (2\lfloor L/(2h^\star)\rfloor)^{-1}$, the only possible minimizers of $\FzL$ are $E_{h^\pm}$. Since $|h^\pm-h^\star|\les L^{-1}$, this concludes the proof of \eqref{deltah}.

\end{proof}
\begin{remark}\label{remhunique}
   From the proof of \eqref{deltah}, it is not hard to see that for most values of $L$, the minimizer of $\FzL$ is actually unique and equal to $E_{h^+}$ or $E_{h^-}$.
\end{remark}
We now turn to the proof of \eqref{chessboardineq}. The idea is, as in \cite{2014CMaPh.tmp..127G,MR2864796}, to use the method of reflection positivity. As in these papers (which we mostly follow), the main point is to prove it for the non-local part of the energy. 
However, we face here the slight technical difficulty that the kernel $|s|^{-q}$ is not integrable around zero and thus we cannot directly split
the integral in \eqref{energi1D} into two pieces but will use  the Laplace transform first.
\begin{lemma}
   Let  $\rho\ge 0$ be such that $\int_0^{+\infty} \rho =1$ and let 
   \[
      \hat\rho(\alpha):= -\rho(\alpha)+\frac{2C_q\alpha^{q-3}}{\Gamma(q) },
   \]
   where $\Gamma$ is Euler's Gamma function. Then, for every $L-$periodic set $E$ of finite perimeter,
   \begin{equation}\label{energie1dLapl}
      \FzL(E)=\int_0^{+\infty} \frac{1}{L}\lt(\hat \rho(\alpha) \per(E,[0,L)) -\frac{C_q \alpha^{q-1}}{\Gamma(q)} \int_{[0,L]\times\R} |\chi_{E}(x)-\chi_{E}(y)|e^{-\alpha |x-y|}\rt) d\alpha.
   \end{equation}
\end{lemma}
\begin{proof}
   Since for $s>0$, $s^{-q}=\frac{1}{\Gamma(q)} \int_{0}^{+\infty} \alpha^{q-1} e^{-\alpha s}$,  we have 
   \begin{multline*}
    \FzL(E)=\frac{1}{L}\lt(\int_0^{+\infty}-\rho \per(E,[0,L))\rt.\\
\lt.    +\frac{C_q}{\Gamma(q)} \int_{\R}\int_0^{+\infty} \alpha^{q-1}e^{-\alpha|z|}\lt[\per(E,[0,L))|z|-\int_0^L|\chi_E(x)-\chi_E(x+z)|\rt]\rt).  
   \end{multline*}
The set $E$ being of finite perimeter, it is a finite union of intervals from which arguing as in the proof of \eqref{eq:basicComparison21} we see that the function $\alpha^{q-1}e^{-\alpha|z|}\lt[\per(E,[0,L))|z|-\int_0^L|\chi_E(x)-\chi_E(x+z)|\rt]$ is integrable in $(\alpha,z)$ so that we can apply Fubini to obtain
\begin{multline*}
    \FzL(E)=\frac{1}{L}\lt(\int_0^{+\infty}-\rho \per(E,[0,L))\rt.\\
\lt.    +\frac{C_q}{\Gamma(q)}  \alpha^{q-1}\int_{\R} e^{-\alpha|z|}\lt[\per(E,[0,L))|z|-\int_0^L|\chi_E(x)-\chi_E(x+z)|\rt]\rt).  
   \end{multline*}
   Using that $\int_{\R} |z|e^{-\alpha |z|} dz=2 \alpha^{-2}$ we conclude the proof of \eqref{energie1dLapl}.
   \end{proof}

For $\alpha,h>0$, let 
\[e_{\alpha,\infty}(h):=-\frac{1}{2h}\int_{0}^{2h} \int_{\R}|\chi_{E_h}(x)-\chi_{E_h}(y)| e^{-\alpha |x-y|}=\lim_{L\to +\infty} -\frac{1}{L} \int_{[0,L]\times \R} |\chi_{E_h}(x) -\chi_{E_h}(y)| e^{-\alpha|x-y|}.\]
Up to noticing that in \eqref{chessboardineq}, the interfacial terms are the same on both sides, thanks to \eqref{energie1dLapl} and integration in $\alpha$, Lemma~\ref{chessboardlem} is proven provided we can show
\begin{lemma}
   For every $\alpha>0$ and every $L-$periodic set $E$ of finite perimeter,
   \begin{equation}\label{chessboardLapl}
      -\int_{[0,L]\times\R} |\chi_{E}(x)-\chi_{E}(y)|e^{-\alpha |x-y|}dx dy\ge \frac{1}{2}\sum_{x\in \partial E\cap [0,L)} h(x) e_{\alpha,\infty}(h(x))+g(x) e_{\alpha,\infty}(g(x)). 
   \end{equation}

\end{lemma}
\begin{proof}
   Since $\alpha$ is fixed, in order to lighten notation, we will assume that $\alpha=1$.\\

   As pointed out in \cite[Ap. A]{giul_lieb_lebo_stripeddipole}, periodic boundary conditions
   are not well suited for the application of reflection positivity.  We will thus prove a statement similar to \eqref{chessboardLapl} under free boundary conditions. For this, we notice that since the kernel
   $e^{-|s|}$ is integrable and since $E$ is periodic,
   \[
      -\int_{[0,L]\times\R} |\chi_{E}(x)-\chi_{E}(y)|e^{- |x-y|}dx dy=\lim_{k\to +\infty} -\frac1k\int_{[0,kL]\times[0,kL]} |\chi_E(x)- \chi_{E}(y)| e^{-|x-y|}.
   \]
   We are thus left to prove that for every set $E\subset[0,L)$ of finite perimeter, 
   \begin{equation}\label{chessboardLapl2}
      -\int_{[0,L]^2} |\chi_E(x) -\chi_{E}(y)| e^{-|x-y|}\ge \frac{1}{2}\sum_{x\in \partial E} h(x)e_{1,\infty}(h(x))+  g(x) e_{1,\infty} (g(x)).
   \end{equation}
   In order to prove \eqref{chessboardLapl2} we need to introduce some further notation. For $L_1,L_2>0$, and two sets  $E_1\subset[0,L_1)$, $E_2\subset (L_1,L_1+L_2)$,
   we let $L:=L_1+L_2$, $(E_1,E_2):=E_1\cup E_2$ and 
   \[\mathcal{J}(E_1,E_2):=-\int_{[0,L]\times[0,L]} |\chi_{(E_1,E_2)}(x)-\chi_{(E_1,E_2)}(y)|e^{- |x-y|}.\]
   We then define the set $(E_1,\theta E_1)$ in $[0,2L_1]$ by  
   \[\chi_{(E_1,\theta E_1)}(x):=\begin{cases}
         \chi_{E_1}(x) & \textrm{for } x\in[0,L_1]\\
         1-\chi_{E_1}(2L_1-x)& \textrm{for } x\in (L_1,2L_1].
      \end{cases}\]
   Letting $L_2=L-L_1$, we similarly define, $(\theta E_2, E_2)$ as a subset of $[L_1-L_2,L]$ by 
   \[\chi_{(\theta E_2, E_2)}(x):=\begin{cases}
         \chi_{E_2}(x) & \textrm{for } x\in[L_1,L]\\
         1-\chi_{E_2}(2L_1-x)& \textrm{for } x\in (L_1-L_2,L_1].
      \end{cases}\]
   The key estimate is
   \begin{equation}\label{keychess}
      \mathcal{J}(E_1,E_2)\ge \frac{1}{2} \lt( \mathcal{J}(E_1,\theta E_1)+\mathcal{J}(\theta E_2,E_2)\rt).
   \end{equation}
   Once \eqref{keychess} is established, \eqref{chessboardLapl2} follows by multiple reflections. We refer the reader to \cite[Lem. A.1]{MR2864796}  or to \cite[Ap. A]{giul_lieb_lebo_stripeddipole}
   for instance for a proof of this fact.\\
   Let us prove \eqref{keychess}. We start by computing $\mathcal{J}(E_1,E_2)$. By definition,
   \begin{align*}
      \mathcal{J}(E_1,E_2)=&-\int_{[0,L_1]^2} |\chi_{E_1}(x)-\chi_{E_1}(y)| e^{-|x-y|}-\int_{[L_1,L]^2} |\chi_{E_2}(x)-\chi_{E_2}(y)| e^{-|x-y|}\\
      &-2\int_{[0,L_1]\times[L_1,L]} |\chi_{E_1}(x)-\chi_{E_2}(y)|e^{-|x-y|}.
   \end{align*}

   Using that  $ |\chi_{E_1}(x)-\chi_{E_2}(y)|=\chi_{E_1}(x)\chi_{E_2^c}(y)+\chi_{E^c_1}(x)\chi_{E_2}(y)$ and that for  $x\in [0,L_1]$ and $y\in[L_1,L]$, $|x-y|=y-x$, we obtain
   \begin{align*}
      \mathcal{J}(E_1,E_2)=&-\int_{[0,L_1]^2} |\chi_{E_1}(x)-\chi_{E_1}(y)| e^{-|x-y|}-\int_{[L_1,L]^2} |\chi_{E_2}(x)-\chi_{E_2}(y)| e^{-|x-y|}\\
      &-2\lt(\int_{0}^{L_1} \chi_{E_1}e^x\rt)\lt(\int_{L_1}^L\chi_{E_2^c}e^{-y}\rt)-2\lt(\int_{0}^{L_1} \chi_{E^c_1}e^x\rt)\lt(\int_{L_1}^L\chi_{E_2}e^{-y}\rt).
   \end{align*}
   Using the definition of $(E_1,\theta E_1)$, we compute similarly
   \begin{align*}
      \mathcal{J}(E_1,\theta E_1)=&-\int_{[0,L_1]^2} |\chi_{E_1}(x)-\chi_{E_1}(y)| e^{-|x-y|}-\int_{[L_1,2 L_1]} |\chi_{\theta E_1}(x)-\chi_{(\theta E_1)}(y)| e^{-|x-y|}\\
      & \ -2\int_{[0,L_1]\times[L_1,2L_1]} (\chi_{E_1}(x) \chi_{(\theta E_1)^c}(y) +\chi_{E^c_1}(x) \chi_{\theta E_1}(y))e^{x-y}\\
      =&-2\int_{[0,L_1]^2} |\chi_{E_1}(x)-\chi_{E_1}(y)| e^{-|x-y|}\\
      &\ -2\int_{[0,L_1]\times[L_1,2L_1]}(\chi_{E_1}(x)\chi_{E_1}(2L_1-y)+\chi_{E_1^c}(x)\chi_{E_1^c}(2L_1-y))e^{x-y}\\
      =&-2\int_{[0,L_1]^2} |\chi_{E_1}(x)-\chi_{E_1}(y)| e^{-|x-y|}-2e^{-2L_1}\lt[\lt(\int_0^{L_1} \chi_{E_1} e^x\rt)^2 +\lt(\int_0^{L_1} \chi_{E^c_1} e^x\rt)^2\rt].
   \end{align*}
   Analogously, we get 
   \[
      \mathcal{J}(\theta E_2, E_2)=-2\int_{[L_1,L]^2} |\chi_{E_2}(x)-\chi_{E_2}(y)| e^{-|x-y|}-2e^{2L_1}\lt[\lt(\int_{L_1}^{L} \chi_{E_2} e^{-x}\rt)^2 +\lt(\int_{L_1}^{L} \chi_{E^c_2} e^{-x}\rt)^2\rt].
   \]
   This concludes the proof of \eqref{keychess} since 
   \begin{multline*}
      e^{2L_1}\lt[\lt(\int_{L_1}^{L} \chi_{E_2} e^{-x}\rt)^2 +\lt(\int_{L_1}^{L} \chi_{E^c_2} e^{-x}\rt)^2\rt]+e^{-2L_1}\lt[\lt(\int_0^{L_1} \chi_{E_1} e^x\rt)^2 +\lt(\int_0^{L_1} \chi_{E^c_1} e^x\rt)^2\rt]\ge\\ 
      2\lt(\int_{0}^{L_1} \chi_{E_1}e^x\rt)\lt(\int_{L_1}^L\chi_{E_2^c}e^{-x}\rt)+2\lt(\int_{0}^{L_1} \chi_{E^c_1}e^x\rt)\lt(\int_{L_1}^L\chi_{E_2}e^{-x}\rt).
   \end{multline*}
\end{proof}

\section*{Acknowledgment}
M. Goldman thanks X. Blanc, M. Lewin and B. Merlet for stimulating discussions.
\bibliography{stripes}
\bibliographystyle{plain}
\end{document}